\documentclass[10pt]{amsart}

\usepackage{amsmath, amsfonts, amsthm, amssymb, graphicx, fullpage, enumerate}

\newtheorem{theorem}{Theorem}     
\newtheorem{lemma}{Lemma}
\newtheorem{corollary}{Corollary}
\newtheorem{definition}{Definition}
\theoremstyle{remark}
\newtheorem*{rem}{Remark}
  
\def\N{\mathbb{N}}

\def\R{\mathbb{R}}
\def\Z{\mathbb{Z}}

\def\P{\mathcal{P}}
\def\S{\mathfrak{S}}

\begin{document}
\title{Polynomial differences in the primes}
\author{Neil Lyall\quad\quad\quad Alex Rice}

\address{Department of Mathematics, The University of Georgia, Athens, GA 30602, USA}
\email{lyall@math.uga.edu}
\address{Department of Mathematics, The University of Georgia, Athens, GA 30602, USA}
\email{arice@math.uga.edu}

\subjclass[2000]{11P05, 11P32, 11P55}

\begin{abstract} We establish, utilizing the Hardy-Littlewood Circle Method, an asymptotic formula for the number of pairs of primes whose differences lie in the image of a fixed polynomial. 
We also include a generalization of this result where differences are replaced with any integer linear combination of two primes.
\end{abstract}\maketitle

\setlength{\parskip}{3pt}

\section{Introduction}

Given a natural number $N$, how many pairs of primes less than or equal to $N$ differ by a perfect square? 

More generally, for an arbitrary polynomial $f \in \Z[x]$, we define 
\begin{equation} \label{rfdef}  r_f(N)=\# \{ (p_1,p_2) \in \mathcal{P}_N^2 \,:\,p_1-p_2 \in f(\N) \},
\end{equation} 
where $\P$ denotes the primes and $\P_N = \P \cap \{1,\dots,N\}$. In this article we will first provide a heuristic, then argue rigorously with the Hardy-Littlewood Circle Method, for the following result.

\begin{theorem} \label{T1} If $f(x) = c_kx^k +\cdots+c_1x +c_0 \in \Z[x]$, $c_k>0$, and $k \geq 1$, then 
\begin{equation} \label{rf} 
r_f(N) = \prod_{p \in \mathcal{P}} \left(1+\frac{z_f(p)-1}{(p-1)^2}\right)\frac{1}{c_k^{1/k}}\frac{k}{k+1} \frac{N^{(k+1)/k}}{\log^2 N} + O\left(\frac{N^{(k+1)/k}\log \log N}{\log^3 N}\right),
 \end{equation} 
where $z_f(p)$ denotes the number of roots of $f$ modulo $p$, and the implied constant depends only on $f$.  \end{theorem}
 
We note that  due to the sufficient decay of $(z_f(p)-1)/(p-1)^2$, the product in the leading term is always finite, and is only 0 if one of the individual terms is 0. This only occurs if $z_f(2)=0$, yielding a meaningful result as long as $f$ has a root modulo $2$. Also, if $f(x)=cx^k$, then the product collapses to $\prod_{p \mid c} \frac{p}{p-1}$, including its complete disappearance if $f(x)=x^k$. 

\begin{rem} The condition that $f$ has a root modulo $2$ is one piece of the more stringent condition that $f$ is an ``intersective polynomial", i.e. $f$ has a root modulo $n$ for all $n \in \N$. This additional assumption on $f$ is necessary and sufficient to conclude the same order of growth for the count analogous to $r_f(N)$ for any subset of the primes of positive relative upper density, shown in a recent result of Th\'ai Ho\`ang L\'e  \cite{Le}. 
\end{rem}

While there is an extensive literature devoted to questions of this type going back some 60 or 70 years, namely Goldbach type problems for polynomials (see for example \cite{Ha}, \cite{Zu}, \cite{Pe} and \cite{BKW}), it appears that the precise result stated in Theorem \ref{T1} above has not been considered previously. 

\section{Heuristics}
We begin by addressing the original motivating question, temporarily letting $r(N)=r_{x^2}(N)$. If we consider the discrete derivative $r'(N) = r(N)-r(N-1)$, we see that $r'(N)=0$ unless $N \in \P$, which by the Prime Number Theorem occurs with probability $1/\log N$. 
If $N$ does happen to be prime, then it induces $|\P_{N-1}| \sim N/\log N$ new positive differences of primes, which are each perfect squares with probability $1/\sqrt{N}$. 

Therefore, we may well expect $r'(N)$ to grow on average like $\sqrt{N}/\log^2N$, which in turns leads, via a simple summation by parts argument, to the prediction that 
\begin{equation}
r(N) = \sum_{k=1}^N r'(k) = \frac{2}{3}\frac{N^{3/2}}{\log^2N} + \text{error}.\end{equation}
This heuristic immediately generalizes to predict 
\begin{equation}\label{xk}
r_{x^k}(N) = \frac{k}{k+1}\frac{N^{(k+1)/k}}{\log^2 N} + \text{error},
\end{equation}
but it fails to account for potential congruence biases. 

In the spirit of generalization, we fix $f(x)= c_kx^k +\cdots+c_1x +c_0 \in \Z[x]$. Note that $r_f(N)=r_{-f}(N)$, so we can assume without loss of generality that $c_k>0$. We must consider how $r_f(N)$ should compare to the prediction (\ref{xk}) above. One difference is that the leading coefficient $c_k$ causes the number of points in the image of $f$ that are at most $N$ to be reduced asymptotically by a factor of $c_k^{-1/k}$. We must also consider local congruence biases of both the primes and the polynomial.
For example, if we took $f(x)=2x^k+1$, then the second prime in each pair would have to be 2, $r_f(N)$ would clearly then be bounded by $|\P_N| \sim N/\log N$, and we couldn't possibly obtain the same order of growth as predicted in (\ref{xk}). The reason for this collapse is of course that almost all primes are odd, so almost no differences of primes are odd. More generally, we can use knowledge about the distribution of primes to gain knowledge about the distribution of differences of primes, and use these facts to investigate which polynomials these differences favor and avoid.

In order to exhaust all possible congruence biases, we must consider the distribution of differences of primes and the image of our polynomial modulo $n$ for each $n \in \N$. However, if two moduli $n$ and $m$ are coprime, then congruence modulo $n$ and congruence modulo $m$ are independent events (by the Chinese Remainder Theorem). Therefore, we can determine all bias by investigating modulo arbitrarily large powers of each prime. Toward this end, we consider an arbitrary prime $p$, and we proceed probabilistically. We note that if two sets $A$ and $B$ are uniformly distributed modulo $p^n$, and we randomly select $a \in A$ and $b \in B$, then the probability that $a \equiv b$ modulo $p^n$, which we denote by $\textbf{P}(a \equiv b$  mod $p^n)$, is $1/p^n$. We compare the analogous probability for our specific sets with this expectation by defining a ``bias factor" 
\begin{equation}
b_f(p^n)= \frac{\textbf{P}(p_1-p_2 \equiv f(d) \text{ mod }p^n)}{1/p^n}\end{equation}
and further noting that
\begin{align*}
b_f(p)&= p^n\sum_{a=0}^{p^n-1} \textbf{P}(p_1-p_2 \equiv a \text{ mod } p^n)\textbf{P}(f(d) \equiv a \text{ mod } p^n)\\
&=\sum_{a=0}^{p^n-1} \textbf{P}(p_1-p_2 \equiv a \text{ mod } p^n) (\# \text{ solutions to } f(x)=a \text{ in } \Z / p^n\Z).
\end{align*}  


Once we determine these biases, we can make a prediction of the form 
\begin{equation}\label{first}
r_f(N)=\prod_{p \in \mathcal{P}} \limsup_{n \to \infty} b_f(p^n)\, \frac{1}{c_k^{1/k}} \frac{k}{k+1}\frac{N^{(k+1)/k}}{\log^2 N} + \text{error}.
\end{equation}
This formulation turns out to be unnecessarily frightening. The Prime Number Theorem for Arithmetic Progressions due to Siegel and Walfisz states that for any modulus $m$, the primes are evenly distributed among the congruence classes coprime to $m$. Therefore, the only biases of the primes, and hence the only biases of differences of primes, are related the coprimality. However, we know that $(a,p)=1$ if and only if $(a,p^n)=1$ for all $n$, which means that $\textbf{P}(p_1-p_2 \equiv a \text{ mod } p^n)=\textbf{P}(p_1-p_2 \equiv b \text{ mod } p^n)$ if $a \equiv b$ mod $p$. One can easily show from the definition above that this implies $b_f(p^n)=b_f(p)$ for all $n$, greatly simplifying prediction (\ref{first}) to 
\begin{equation}
r_f(N)=\prod_{p \in \mathcal{P}} b_f(p) \frac{1}{c_k^{1/k}}\frac{k}{k+1}\frac{N^{(k+1)/k}}{\log^2 N} + \text{error}.
\end{equation}

 Now, for each fixed prime $p$, it follows from the Siegel-Walfisz Theorem that primes are congruent to $1,2,\dots,p-1$ modulo $p$ each with probability $1/(p-1)$ and consequently that the difference of two randomly selected primes is congruent to $0$ modulo $p$ with probability $1/(p-1)$, and congruent to $1,2,\dots,p-1$ modulo $p$ each with probability $(p-2)/(p-1)^2$. 
From this observation it follows that the bias factor $b_f(p)$ is completely determined by the number of roots of $f$ modulo $p$, a quantity we shall denote by $z_f(p)$.  In fact, 
\begin{equation}
b_f(p) = \frac{1}{p-1} z_f(p) + \frac{p-2}{(p-1)^2}\left(p-z_f(p)\right) = 1 +\frac{z_f(p)-1}{(p-1)^2}.
\end{equation}

This completes the heuristic for the result in Theorem \ref{T1}, and indicates that the only ``fatal" obstruction toward the expected order of growth for $r_f(N)$ is the modulo $2$ consideration noted above.

\section{Rigorous Treatment via the Hardy-Littlewood Circle Method}

In this section we give a proof of Theorem \ref{T1} using the circle method developed by Hardy and Littlewood. As is standard, we begin by weighting the characteristic function of the primes with a logarithm to obtain a more uniform distribution. This yields a weighted count intimately related to $r_f(N)$, which we define below using a truncated von Mangoldt function.

\begin{definition} We define $\Lambda_N : \Z \to [0,\infty)$ by 
\begin{equation}
\Lambda_N(n) = \begin{cases} \log p &\text{if $n=p^k \leq N, p \in \P, k \in \N$}\\ 0 &\text{else} \end{cases},
\end{equation}
and for $f \in \Z[x]$ as in Theorem 1, we define 
\begin{equation}
R_f(N) = \sum_{d=1}^M \sum_{n \in \Z} \Lambda_N(n)\Lambda_N(n-f(d)),
\end{equation}
where $M=\left(N/c_k\right)^{1/k}$.
\end{definition}
The main result of this section, and indeed the whole article, is the following.

\begin{theorem}\label{T2} For $f \in \Z[x]$ as in Theorem 1 and any $A>0$, we have 
\begin{equation}
R_f(N) = \mathfrak{S}(f) \frac{1}{c_k^{1/k}} \frac{k}{k+1} N^{(k+1)/k} + O\left(\frac{N^{(k+1)/k}}{\log^A N}\right),
\end{equation}
where the implied constant depends only on $f$ and $A$, and
\begin{equation}
\S(f) = \sum_{q=1}^{\infty} \frac{\mu(q)^2}{q\phi(q)^2}\sum_{\substack{0 \leq a < q \\ (a,q)=1}} \sum_{r=0}^{q-1} e^{2\pi \text{i} f(r) a/q},
\end{equation} 
where $\mu$ is the M\"obius function, $\phi$ is the Euler totient function.
\end{theorem}

It is easy to see that Theorem \ref{T1} will follow from Theorem \ref{T2} and the following two lemmas.

\begin{lemma}\label{SSLem} If $f \in \Z[x]$ is as in Theorem \ref{T1} and $z_f(p)$ denotes the number of roots of $f$ modulo $p$, then
\begin{equation}
\sum_{q=1}^{\infty} \frac{\mu(q)^2}{q\phi(q)^2}\sum_{\substack{0 \leq a < q \\ (a,q)=1}} \sum_{r=0}^{q-1} e^{2\pi \text{i} f(r) a/q} =  \prod_{p \in \mathcal{P}} \left(1+\frac{z_f(p)-1}{(p-1)^2}\right).
\end{equation}
\end{lemma}

\begin{lemma} \label{Wlem} If $f \in \Z[x]$ is as in Theorem \ref{T1}, then
\begin{equation}
r_f(N) = \frac{R_f(N)}{\log^2 N} +  O\left(\frac{N^{(k+1)/k} \log \log N}{\log^3 N}\right).
\end{equation}
\end{lemma}

The proof of both Lemma \ref{SSLem} and Lemma \ref{Wlem} are standard exercises, for completeness however we include their proofs in Section \ref{4} below. 
In the remainder of this section we present the proof of Theorem \ref{T2}.


\subsection{Proof of Theorem \ref{T2}}

We fix $f$ as in Theorem \ref{T1} and $A>0$.

In order to apply the circle method, we need to reformulate the weighted count $R_f(N)$ on the transform side as an integral over the circle. Specifically, it follows from the usual orthogonality relation 
\begin{equation}
\int_0^1 e^{2\pi in\alpha}\,d\alpha= \begin{cases}1 &\text{if $n=0$}\\ 0 &\text{if $n\in\Z\setminus\{0\}$} \end{cases}
\end{equation}
that 
\begin{equation}\label{trans} 
R_f(N) = \int_0^1 |\widehat{\Lambda_N}(\alpha)|^2 S_M(\alpha)\, d\alpha, \end{equation} 
where  
\begin{equation}
\widehat{\Lambda_N}(\alpha) = \sum_{n \in \Z} \Lambda_N(n) e^{-2\pi in\alpha}
\end{equation}
and
\begin{equation}
S_M(\alpha)=\sum_{d=1}^M e^{2\pi if(d)\alpha}.
\end{equation}

As usual, we wish to partition the circle into two pieces, one the collection of points near rationals with small denominator, from which the primary contribution to the integral (\ref{trans}) will stem, and the other simply the complement. In order to appropriately make this partition official, we need a parameter yielded from the following estimate on high moments of the Weyl sum $S_M$.
\begin{lemma} \label{Wsum} There exists $s_0(k)=O(k^2\log k)$ such that if $s \geq s_0(k)$, then 
\begin{equation}
\int_0^1 |S_M(\alpha)|^s\,d\alpha = O(M^{s-k}).
\end{equation}
\end{lemma} 
For a proof of this estimate see Proposition 3.3 in \cite{NL}. We remark that it follows from recent work of Trevor Wooley \cite{W}, that one can in fact take $s_0(k)=2k(k+1)$.

\begin{definition} Fixing $s= \max \{s_0(k), 4k \}$ we let $B=s(A+1)+8$ and define $\mathfrak{M}$, the \textbf{major arcs}, by  
\begin{equation}
\mathfrak{M}=\bigcup_{q=1}^{\log^BN} \bigcup_{\substack{0 \leq a < q\\ (a,q)=1}}\mathbf{M}_{a/q},
\end{equation}
where
\begin{equation}
\mathbf{M}_{a/q}=\left\{\alpha \in [0,1] \,:\,\Big|\alpha - \frac{a}{q}\Big|< \frac{\log^BN}{N} \right\},
\end{equation} 
and $\mathfrak{m}$, the \textbf{minor arcs}, by $\mathfrak{m}=[0,1]\setminus \mathfrak{M}.$
\end{definition}

\subsubsection{Estimates on the minor arcs}
We first focus our attention on the minor arcs, in an attempt to absorb their contribution to the integral into the error term in Theorem \ref{T2}. To this end, we invoke the minor arc estimate on $\widehat{\Lambda_N}$ that was Vinogradov's main achievement in solving the ternary Goldbach problem unconditionally. 

\begin{lemma}[Vinogradov, Vaughan] \label{Vmin} If $\alpha \in [0,1]$, $|\alpha-a/q|<1/q^2$ and $(a,q)=1$, then 
\begin{equation}\widehat{\Lambda_N}(\alpha) = O\left(\log^4 N\left(N/q^{1/2}+ N^{4/5} + N^{1/2}q^{1/2}\right)\right).
\end{equation}
\end{lemma}
The argument for this estimate was greatly simplified by Vaughan, and an exposition of it can be found in \cite{Vaughan}, for example.
It is now a relatively straightforward matter to see that  Lemmas \ref{Wsum} and \ref{Vmin} combine to give the following desired estimate for our integral over the minor arcs.

\begin{corollary}[Minor arc estimate] \label{min} 
\begin{equation}
\int_{\mathfrak{m}}  |\widehat{\Lambda_N}(\alpha)|^2 S_M(\alpha)\,d\alpha = O\left(\frac{N^{(k+1)/k}}{\log^A N}\right).
\end{equation}
\end{corollary}

\begin{proof} First we fix an arbitrary $\alpha \in \mathfrak{m}$. By the pigeonhole principle, there must exist $1 \leq q \leq N/\log^B N$ such that 
\[\Big|\alpha - \frac{a}{q}\Big| < \frac{\log^B N}{qN} \leq \frac{1}{q^2}.\] 

However, by the definition of $\mathfrak{m}$, it must be the case that $q \geq \log^B N$. Combining these bounds on $q$ with Lemma \ref{Vmin} gives \begin{equation}\label{Vmin2} 
\widehat{\Lambda_N}(\alpha) = O\left(\frac{N}{\log^{(B-8)/2}N}\right). \end{equation}

Now, we invoke H\"older's Inequality in order to utilize our control on the higher moments of $S_M$. Indeed, applying H\"older's inequality, Lemma \ref{Wsum} and (\ref{Vmin2}) we obtain
\begin{align*} \left|\int_{\mathfrak{m}}  |\widehat{\Lambda_N}(\alpha)|^2 S_M(\alpha)\, d\alpha \right| &\leq \left( \int_{\mathfrak{m}} |\widehat{\Lambda_N}(\alpha)|^{2s/(s-1)}\,d\alpha \right)^{(s-1)/s}\left(\int_0^1 |S_M(\alpha)|^s\,d\alpha\right)^{1/s}\\
&= O\left(\left(\frac{N}{\log^{(B-8)/2} N}\right)^{2/s} 
\left( \int_0^1 |\widehat{\Lambda_N}(\alpha)|^{2}\,d\alpha \right)^{(s-1)/s}
M^{\frac{s-k}{s}} \right)\\
&=O\left(\frac{N^{(k+1)/k}}{\log^A N}\right).
\end{align*} 

In the last line above we have used the fact that $M=O(N^{1/k})$ and $B=s(A+1) + 8$, together with the deeper fact that
\[\int_0^1 |\widehat{\Lambda_N}(\alpha)|^{2}\,d\alpha=O(N\log N),\]
which is a standard consequence of the Prime Number Theorem (and Plancherel's Identity). 
\end{proof}

\subsubsection{Estimates on the major arcs}
For the major arcs, we invoke the most classical of estimates.
\begin{lemma}[Major arc estimates] \label{maj} If $\alpha = a/q+ \beta$ with $1 \leq q \leq \log^B N$, $(a,q)=1$ and $|\beta| < \log^B N/N$, then 
\begin{equation} \label{SM} 
S_M(\alpha)= q^{-1} S_{a/q} I_M(\beta) + O(\log^{2B} N)
\end{equation} 
and  
\begin{equation}\label{Lh}
\widehat{\Lambda_N}(\alpha) = \frac{\mu(q)}{\phi(q)} \nu_N(\beta) + O(Ne^{-c\sqrt{\log N}})\end{equation} 
for some $c=c(B)>0$, where $\mu$ is the M\"obius function, $\phi$ is the Euler totient function, 
\[S_{a/q} = \sum_{r=0}^{q-1} e^{2 \pi if(r)a/q},\quad I_M(\beta) = \int_0^M e^{2\pi if(x)\beta}\,dx,\quad\text{and}\quad \nu_N(\beta) = \sum_{n=1}^N e^{-2\pi in\beta}.\]
\end{lemma}
Estimate (\ref{SM}) can be found in any discussion of Waring's Problem (Lemma 4.2 of \cite{Dav1} for example), while estimate (\ref{Lh}) follows from the Siegel-Walfisz Theorem on primes in arithmetic progressions, and is included in any discussion of Vinogradov's Three Primes Theorem (see for example Lemma 3.1 in \cite{Vaughan}). We further note that
\begin{equation} \label{nu} \nu_N(\beta)=\int_0^N e^{-2\pi ix\beta}dx + O(\log^B N),
\end{equation} 
thus we maintain the same asymptotic by replacing the sum with the integral. 

We are now ready to combine all of our estimates and establish Theorem \ref{T2}.

\subsubsection{Proof of Theorem \ref{T2}} From  (\ref{trans}) and Corollary \ref{min}, and noting that for large $N$ the individual major arcs are pairwise disjoint, we have 
\begin{equation}
R_f(N) = \sum_{q=1}^{\log^B N}\sum_{\substack{0 \leq a < q \\ (a,q)=1}}  \int_{\mathbf{M}_{a/q}}|\widehat{\Lambda_N}(\alpha)|^2  S_M(\alpha)\, d\alpha + O\left(\frac{N^{(k+1)/k}}{\log^A N}\right).
\end{equation}
From (\ref{SM}), (\ref{Lh}), and (\ref{nu}), we then obtain 
\begin{equation}\label{R1} 
R_f(N)= \underbrace{\sum_{q=1}^{\log^B N}\frac{\mu(q)^2}{q\phi(q)^2} \sum_{\substack{0 \leq a < q \\ (a,q)=1}} S_{a/q}}_{(\star)} \underbrace{\int\limits_{|\beta|<(\log^B N)/N}\left|\int_0^N e^{-2\pi iy\beta}\,dy\right|^2  I_M(\beta)\, d\beta}_{(\star\star)}+O\left(\frac{N^{(k+1)/k}}{\log^A N}\right). 
\end{equation}

The reader can now easily see that Theorem \ref{T2} will be an immediate consequence of the following estimates for $(\star)$ and $(\star\star)$ respectively:
\begin{equation}\label{ss}
\sum_{q=1}^{\log^B N}\frac{\mu(q)^2}{q\phi(q)^2} \sum_{\substack{0 \leq a < q \\ (a,q)=1}} S_{a/q}=\sum_{q=1}^{\infty} \frac{\mu(q)^2}{q\phi(q)^2}\sum_{\substack{0 \leq a < q \\ (a,q)=1}} S_{a/q}+ O\left(\frac{1}{\log^{A} N}\right)
\end{equation}
and
\begin{equation}\label{si}
\int\limits_{|\beta|<(\log^B N)/N}\left|\int_0^N e^{-2\pi iy\beta}\,dy\right|^2  I_M(\beta)\, d\beta=\frac{k}{k+1}MN+O\left(\frac{MN}{\log^BN}  \right).
\end{equation}

In order to establish (\ref{ss}) we need an estimate on the magnitude of the Gauss sum $S_{a/q}$ that beats the trivial bound of $q$. 
Theorem 7.1 in \cite{Vaughan} tells us that 
\begin{equation}\label{gsum}S_{a/q} = O_{\epsilon}(q^{(1-1/k)+\epsilon}) \text{ for all } \epsilon > 0, 
\end{equation}
so in particular, we know that $S_{a/q} = O(q^{1-1/2k})$, and since $\phi(q) \geq Cq^{1-\frac{1}{4k}}$ (trivially), we can deduce from this that
\begin{align*}
\Big|\sum_{q=\log^B N}^{\infty}\frac{\mu(q)^2}{q\phi(q)^2} \sum_{\substack{0 \leq a < q \\ (a,q)=1}} S_{a/q}\Big| &= O\left(\sum_{q=\log^B N}^{\infty} \frac{1}{q^{1+1/4k}}\right)\\
& = O\left(\frac{1}{\log^{B/4k}N}\right).
\end{align*}
Noting that $B/4k> A$, we see that (\ref{ss}) immediately follows.


In order to establish (\ref{si}) we initially note that 
\begin{equation} \label{IM1} 
I_M(\beta) - \int_0^M e^{2\pi ic_kx^k\beta}\,dx  = O\left(\int_0^Mx^{k-1}\beta\,dx\right)=O(\log^B N)\end{equation}
and that substituting this into the left hand side of (\ref{si}) gives
\begin{equation*}
\int\limits_{|\beta|<(\log^B N)/N}\!\!\left|\int_0^N e^{-2\pi iy\beta}\,dy\right|^2  I_M(\beta)\, d\beta=\!\!\!\!\!\!\!\!\!\!\!\!\int\limits_{|\beta|<(\log^B N)/N}\!\!\left|\int_0^N e^{-2\pi iy\beta}\,dy\right|^2  \left(\int_0^M e^{2\pi ic_kx^k\beta}dx\right)\, d\beta+O(N\log^{2B} N).
\end{equation*}

We note further that after three changes of variables (namely $x:= x/M$, $\beta := N\beta$, $y := y/N$), the integral on the right hand side of the above identity can be seen to satisfy
\begin{equation} \label{SI1} 
\int\limits_{|\beta|<\log^B N}\!\!\left|\int_0^N e^{-2\pi iy\beta}\,dy\right|^2  \left(\int_0^M e^{2\pi ic_kx^k\beta}dx\right)\, d\beta= MN\!\!\!\!\!\!\!\!\!\int\limits_{|\beta|<\log^B N}\!\!\!\!\!\!|\widehat{1_{[0,1]}}(\beta)|^2\,\left(\int_0^1 e^{2\pi ix^k\beta}dx\right)\, d\beta, \end{equation}\\ where $\widehat{1_{[0,1]}}(\beta) = \int_0^1 e^{-2 \pi iy\beta}\, dy$ is the usual Euclidean Fourier transform of the characteristic function of the unit interval. Utilizing the standard (and easily verified) fact that $|\widehat{1_{[0,1]}}(\beta)|^2 = O(1+|\beta|^2)^{-1}$, it then follows that
 \begin{equation}
 \int\limits_{|\beta|\geq\log^B N}\!\!\!\!\!\!|\widehat{1_{[0,1]}}(\beta)|^2\,\left(\int_0^1 e^{2\pi ix^k\beta}dx\right)\, d\beta = O\left(\frac{1}{\log^{B}N}\right)
 \end{equation}
and hence that
 \begin{equation}
\int\limits_{|\beta|<\log^B N}\!\!\!\!\!\!|\widehat{1_{[0,1]}}(\beta)|^2\,\left(\int_0^1 e^{2\pi ix^k\beta}dx\right)\, d\beta=\int_{\R}|\widehat{1_{[0,1]}}(\beta)|^2\,\left(\int_0^1 e^{2\pi ix^k\beta}dx\right)\, d\beta + O\left(\frac{1}{\log^{B}N}\right).
 \end{equation}

Thus, in order to establish (\ref{si}) it finally suffices to show
\begin{equation}
\int_{\R}|\widehat{1_{[0,1]}}(\beta)|^2\,\left(\int_0^1 e^{2\pi ix^k\beta}dx\right)\, d\beta=\frac{k}{k+1}.
\end{equation}

Changing variables again ($x := x^k$), letting $g(x)=x^{(1-k)/k}/k\,1_{[0,1]} (x)$, and applying Parseval's Identity, it follows that
\begin{align*}
\int_{\R}|\widehat{1_{[0,1]}}(\beta)|^2\,\left(\int_0^1 e^{2\pi ix^k\beta}dx\right)\, d\beta&=\int_{\R}|\widehat{1_{[0,1]}}(\beta)|^2\,\overline{\widehat{g}(\beta)}\,d\beta\\
&=\int_{\R}1_{[0,1]}(x)\int_{\R}g(y)1_{[0,1]}(x-y)\,dy\,dx\\
&=\int_0^1 \int_0^x \frac{y^{(1-k)/k}}{k} \,dy\, dx \\
&=\frac{k}{k+1}
\end{align*}
as required.\qed



\section{Proof of Lemmas \ref{SSLem} and \ref{Wlem}}\label{4}

\subsection{Proof of Lemma \ref{SSLem}}\label{SectionSSlem} First we write 
\begin{equation}
\sum_{q=1}^{\infty} \frac{\mu(q)^2}{q\phi(q)^2}\sum_{\substack{0 \leq a < q \\ (a,q)=1}} \sum_{r=0}^{q-1} e^{2\pi \text{i} f(r) a/q} = \sum_{q=1}^{\infty}F(q),
\end{equation}
 where 
 \begin{equation}
F(q)= \frac{\mu(q)^2}{q\phi(q)^2} \sum_{\substack{0 \leq a < q \\ (a,q)=1}} S_{a/q}\quad\text{and}\quad S_{a/q} = \sum_{r=0}^{q-1} e^{2 \pi if(r)a/q}.
\end{equation}

Since $q$, $\mu(q)$, and $\phi(q)$ are all multiplicative functions, and it is a standard exercise to verify that
\begin{equation}
\sum_{\substack{0 \leq a < q \\ (a,q)=1}} S_{a/q}
\end{equation}
is also multiplicative (see for example Lemma 5.1 of \cite{Dav1} for a proof), it follows that $F(q)$ is a multiplicative function.
Furthermore, during the proof of Theorem \ref{T2}, we used estimate (\ref{gsum}) to establish that $F(q) = O(q^{-1-1/4k})$, so in particular we know that
$\sum_{q=1}^{\infty}F(q)$
is an absolutely convergent sum. 

These two facts combine to yield the usual Euler product formula, namely
\begin{equation}
\sum_{q=1}^{\infty}F(q)= \prod_{p \in \P} \sum_{n=0}^{\infty} F(p^n)= \prod_{p \in \P}(1+F(p)),
\end{equation}
with the great simplification occurring due the presence of the M\"obius function, which forces $F(p^n) = 0$ for every prime $p$, provided $n \geq 2$. 

Noting that 
\begin{equation}
F(p)=\frac{1}{p(p-1)^2} \sum_{a=1}^{p-1} \sum_{r=o}^{p-1} e^{2 \pi i f(r)a/p},
\end{equation}
we see that it now suffices to simply verify that
\begin{equation} \label{z} \sum_{r=0}^{p-1} \sum_{a=1}^{p-1} e^{2 \pi i f(r)a/p}=p(z_f(p)-1)
\end{equation}
for all primes $p$. This is straightforward, since
if $f(r) \equiv 0$ modulo $p$, then the exponential term is identically $1$, and \[\sum_{a=1}^{p-1} e^{2 \pi i f(r)a/p}=p-1.\] 
While if $f(r) \not\equiv 0$ modulo $p$, then the inner sum goes over all $p$-th roots of unity except $1$, hence \[\sum_{a=1}^{p-1} e^{2 \pi i f(r)a/p}=-1.\] Therefore, since $z_f(p)$ denotes the number of roots of $f$ mod $p$, it follows that
\begin{equation}
\sum_{r=0}^{p-1} \sum_{a=1}^{p-1} e^{2 \pi i f(r)a/p}=z_f(p)(p-1)-(p-z_f(p)) = p(z_f(p)-1)
\end{equation}
as required.
\qed


\subsection{Proof of Lemma \ref{Wlem}}\label{SectionWlem} 

First, we let
\begin{equation}
M'= \min \{n \in \N \mid f(d)>N \ \text{for all} \ d >n\},
\end{equation}
so
\begin{equation} 
 \label{a1} \sum_{d=1}^{M'}\sum_{n \in \Z} 1_{\mathcal{P}_N}(n)1_{\mathcal{P}_N}(n-f(d)) = \# \{(p,d) \in \mathcal{P}_N \times \N \,:\, p-f(d) \in \mathcal{P}_N \} = r_f(N) +O\left(\frac{N}{\log N}\right),
 \end{equation}
 where the error term accounts for the finitely many instances of ``double-counting" in the polynomial. This convenience is the primary reason for restricting to the image of $\N$ as opposed to $\Z$, making all nonconstant polynomials ``eventually injective". We also note that $M'=M + O_{\epsilon}(N^{\epsilon})$ for every $\epsilon>0$, and in particular,
\begin{equation} \label{a2}
r_f(N) =  \sum_{d=1}^{M}\sum_{n \in \Z} 1_{\mathcal{P}_N}(n)1_{\mathcal{P}_N}(n-f(d)) + O\left(N^{1+1/4k} \right).
\end{equation}
We now make the usual observation that the contribution to the von Mangoldt function from proper prime powers is negligible. Namely,
\begin{equation} \label{a3}
R_f(N)= \sum_{d=1}^M\sum_{n \in \Z} \log(n)\log(n-f(d))1_{\mathcal{P}_N}(n)1_{\mathcal{P}_N}(n-f(d)) + O\left(N^{(k+2)/2k}\log^2 N\right).
\end{equation}
By (\ref{a1}), (\ref{a2}), and (\ref{a3}), noting the trivial upper bound 
\begin{equation}
\sum_{d=1}^{M}\sum_{n \in \Z} 1_{\mathcal{P}_N}(n)1_{\mathcal{P}_N}(n-f(d)) \geq \frac{1}{\log^2 N}\sum_{d=1}^M\sum_{n \in \Z} \log(n)\log(n-f(d))1_{\mathcal{P}_N}(n)1_{\mathcal{P}_N}(n-f(d)),
\end{equation}
and changing variables $(n:=n-f(d))$, it now suffices to show
\begin{equation}
\begin{split}
\sum_{d=1}^{M}\sum_{n \in \Z} 1_{\mathcal{P}_N}(n)1_{\mathcal{P}_N}(n+f(d)) &\leq \frac{1}{\log^2 N}\sum_{d=1}^M\sum_{n \in \Z} \log(n)\log(n+f(d))1_{\mathcal{P}_N}(n)1_{\mathcal{P}_N}(n+f(d))\\ &+ O\left(\frac{N^{(k+1)/k}\log \log N}{\log^3N} \right). 
\end{split}
\end{equation}
Toward this end, we see that for any $\delta>0$,
\begin{align*} 
\sum_{d=1}^{M}\sum_{n \in \Z} 1_{\mathcal{P}_N}(n)1_{\mathcal{P}_N}(n+f(d)) &= \sum_{d=1}^M \sum_{n=N^{1-\delta}}^N 1_{\mathcal{P}_N}(n)1_{\mathcal{P}_N}(n+f(d)) +  \sum_{d=1}^M\sum_{n=1}^{N^{1-\delta}} 1_{\mathcal{P}_N}(n)1_{\mathcal{P}_N}(n+f(d))\\
& \leq \frac{1}{(1-\delta)^2\log^2 N}  \sum_{d=1}^M\sum_{n\in \Z} \log(n)\log(n+f(d))1_{\mathcal{P}_N}(n)1_{\mathcal{P}_N}(n+f(d)) \\
& +O\left(\frac{N^{(k+1)/k}}{N^{\delta}\log N} \right).
\end{align*}
Setting \[\delta=\frac{2\log\log N}{\log N},\] and noting from Theorem \ref{T2} and (\ref{a3}) that
\begin{equation}
 \sum_{d=1}^M\sum_{n\in \Z} \log(n)\log(n+f(d))1_{\mathcal{P}_N}(n)1_{\mathcal{P}_N}(n+f(d)) = O\left( N^{(k+1)/k} \right),
\end{equation} 
the lemma follows. \qed


\section{Further Generalization}\label{finalsection}

A natural question: How would our results change if, instead of differences of primes, we counted sums of primes that lie in the image of a fixed polynomial? Even more generally, for $f \in \Z[x]$, $N \in \N$, and arbitrary integers $a_1,a_2 \neq 0$, we define
\begin{equation} r_{f,a_1,a_2}(N)=\# \left\{ (p_1,p_2) \in \P_N^2 \,:\,a_1p_1+a_2p_2 \in f(\N) \right\}. \end{equation}

Again we note that $r_{f,a_1,a_2}(N)=r_{-f,-a_1,-a_2}(N)$, so we can assume that $f$ has positive leading coefficient. Under this assumption, if $a_1$ and $a_2$ were both negative, one can see that $r_{f,a_1,a_2}$ would be uniformly bounded in $N$, so we assume without loss of generality that $a_1>0$. 

Careful adaptation of our initial heuristic with the discrete derivative yields an additional factor of
\begin{equation} C_k(a_1,a_2) =\begin{cases} \dfrac{(a_1+a_2)^{1/k}-a_1^{1/k}}{a_2} +\dfrac{(a_1+a_2)^{1/k}-a_2^{1/k}}{a_1} &\text{if }a_2>0\\ \dfrac{(a_1+a_2)^{1/k}-a_1^{1/k}}{a_2}+\dfrac{(a_1+a_2)^{1/k}}{a_1} &\text{if }0>a_2\geq -a_1\\ -\dfrac{a_1^{1/k}}{a_2} &\text{if } a_2 \leq -a_1 \end{cases}.\end{equation}
Note for example that $C_k(1,-1) = 1$ for all $k$, as it should to agree with our previous discussion, and $C_k(1,1)= 2(2^{1/k}-1)$. 

To adapt our heuristic for local congruence biases, it is useful to partition the primes based on how many of the coefficients $a_1,a_2$ they divide. Namely, we define
\begin{align} \P_0&= \{ p \in \P : p \nmid a_1a_2\}\nonumber\\
\P_1&= \{ p \in \P : p \mid a_1a_2,\text{ } (p \nmid a_1 \text{ or } p \nmid a_2) \}\\ \P_2&= \{p \in \P : p \mid a_1, \text{ } p \mid a_2 \}\nonumber. \end{align}
Hopefully context will prevent any confusion between this notation and the notation $\P_N=\P \cap \{1,\dots,N\}$. 

Adaptation of our heuristic with the Siegel-Walfisz Theorem yields local bias factors
\begin{equation} b_{f,a_1,a_2}(p)= \begin{cases}1+ \dfrac{z_f(p)-1}{(p-1)^2}&\text{if }p \in \P_0\\ \dfrac{p-z_f(p)}{p-1}&\text{if }p \in \P_1\\z_f(p)  &\text{if } p \in \P_2 \end{cases}.\end{equation}
Note that we still have $b_{{x^k},a_1,a_2}(p)=1$ for all $p \in \P$, $k \in \N$. 

Indeed, one can establish the following generalization of Theorem \ref{T1}.

\begin{theorem} \label{T3} If $a_1,a_2 \in \Z$, $a_1 >0$, $a_2 \neq 0$, $f(x) = c_kx^k +\cdots+c_1x +c_0 \in \Z[x]$, $c_k>0$, and $k \geq 1$, then 
\begin{equation}
\begin{split}
r_{f,a_1,a_2}(N) = &\prod_{p \in \P_0} \left(1+\frac{z_f(p)-1}{(p-1)^2}\right) \prod_{p \in \P_1}\frac{p-z_f(p)}{p-1} \prod_{p \in \P_2} z_f(p)\,\frac{C_k(a_1,a_2)}{c_k^{1/k}}\frac{k}{k+1} \frac{N^{(k+1)/k}}{\log^2 N}\\ &+ O\left(\frac{N^{(k+1)/k}\log \log N}{\log^3 N}\right),\
\end{split}
\end{equation}
where the implied constant depends on $f, a_1,$ and $a_2$.
\end{theorem}

The vast majority of the argument for Theorem \ref{T3} is immediately analogous to the proof of Theorem \ref{T1}, so rather than starting from scratch, we proceed by highlighting the key differences. Much like (\ref{trans}), we start by giving a weighted, transform side formulation of our count, defining \begin{equation} \label{trans2} R_{f,a_1,a_2} = \int_{0}^1 \widehat{\Lambda_N}(a_1\alpha) \widehat{\Lambda_N}(a_2\alpha) S_M(\alpha) d\alpha, \end{equation} where now \begin{equation}
M= \left(\frac{\max \{a_1, a_1+a_2\}N}{c_k}\right)^{1/k}.
\end{equation}

 The procedure for absorbing the minor arcs into the error term goes through just as before, but some subtlety arises in the major arcs. For example, if $\alpha$ is in some major arc $\textbf{M}_{a/q}$, then (\ref{Lh}) still gives an estimate for $\widehat{\Lambda_N}(a_1 \alpha)$, just potentially with a different denominator, based on the factors shared by $a_1$ and the original denominator $q$. More officially, we decompose each squarefree number based on our previous partition of the primes, i.e. we define
\begin{equation} \N_i = \{p_1p_2\cdots p_\ell : p_j \in \P_i \text{ distinct} \}  \end{equation}
for $i=0,1,2$,
where each set includes $1$, so every squarefree number can be written uniquely as $qm_1m_2$ with $q \in \N_0$, $m_1 \in \N_1$, and $m_2 \in \N_2$. (Recall that due to the presence of the M\"obius function, only major arcs with squarefree denominators contribute to the main term.) 
Now, one can check that by (\ref{Lh}), if $\alpha =a/qm_1m_2 + \beta$, $qm_1m_2 \leq \log^BN$, $(a,qm_1m_2)=1$ and $|\beta| < \log^B N/N$, then \begin{equation} \widehat{\Lambda_N}(a_1 \alpha)\widehat{\Lambda_N}(a_2 \alpha) = \frac{\mu(q)^2\mu(m_1)}{\phi(q)^2\phi(m_1)}\nu_N(a_1\beta)\nu_N(a_2\beta) + O(N^2e^{-c\sqrt{\log n}})  \end{equation}
for some $c=c(B,a_1,a_2)>0$.
 
In place of the integral observed in (\ref{SI1}), we now have \begin{equation}\int_{|\beta|<\log^B N/N}\left(\int_0^N e^{-2\pi ia_1y\beta}\,dy\right) \left(\int_0^N e^{-2\pi ia_2z\beta}\,dz \right) \left(\int_0^M e^{2\pi ic_kx^k\beta}\,dx\right) \,d\beta ,\end{equation} 
which with similar tricks of changing variables and applying Parseval's Identity can be seen to equal 
\begin{equation}
C_k(a_1,a_2)c_k^{-1/k}\frac{k}{k+1} N^{(k+1)/k} + O\left(\frac{N^{(k+1)/k}}{\log^B N}\right).
\end{equation}

In place of the singular series $\S(f)$, we now have  \begin{equation} \S(f,a_1,a_2) = \sum_{\substack{ q \in \N_0 \\ m_1 \in \N_1 \\ m_2 \in \N_2}}\frac{\mu(q)^2\mu(m_1)}{\phi(q)^2\phi(m_1)} \frac{1}{qm_1m_2} \sum_{\substack{0 \leq a < qm_1m_2 \\ (a,qm_1m_2)=1}}S_{a/qm_1m_2},  \end{equation} 
which is still an absolutely convergent sum of a multiplicative function that is zero for non-squarefree numbers, hence we still have the Euler product formula \[\S(f,a_1,a_2) = \prod_{p \in \P} (1+F(p)),\] 
where now 
\begin{equation} F(p) = \begin{cases}\displaystyle{\frac{1}{p(p-1)^2}\sum_{a=0}^{p-1}S_{a/p}}&\text{if }p \in \P_0\\
\displaystyle{-\frac{1}{p(p-1)}\sum_{a=0}^{p-1}S_{a/p}}&\text{if }p \in \P_1\\ \displaystyle{\frac{1}{p} \sum_{a=0}^{p-1}S_{a/p}} &\text{if } p \in \P_2 \end{cases}. \end{equation} 
Recalling our previous calculation that \[\sum_{a=0}^{p-1}S_{a/p} = p(z_f(p)-1),\] the result follows.\qed

\begin{rem} Generalization of this result to linear combinations of more than two primes is definitely possible, but not particularly useful, as the circle method machinery is already capable of tackling the strictly harder task of counting solutions to $a_1p_1+\cdots+a_\ell p_\ell = N$ for fixed $N$ and $\ell \geq 3$. A useful pursuit, however, may be to extend the result to count solutions of the form $a_1p_1^{k_1}+a_2p_2^{k_2}=f(d)$, where $k_1$ and $k_2$ are arbitrary fixed natural numbers.
\end{rem}

\end{document}